\documentclass[11pt]{article}

\usepackage{tabularx}
\usepackage{caption} 
\usepackage{amsmath}
\usepackage{amssymb}
\usepackage{amsthm}
\usepackage{graphicx}
\usepackage{color}
\usepackage{enumitem}
\usepackage{float}
\usepackage{multirow}
\usepackage{tablefootnote}
\usepackage[left=1in,right=1in,top=1in,bottom=1in]{geometry}
\usepackage{xspace}
\usepackage{hyperref}
\usepackage{todonotes}
\usepackage[capitalise]{cleveref}
\usepackage{mathtools}
\usepackage{tikz}
\usetikzlibrary{arrows.meta,shapes.misc}
\usepackage[normalem]{ulem}

\renewcommand{\paragraph}[1]{\vspace{6pt} \noindent \textbf{#1}\xspace}

\numberwithin{equation}{section}
\newtheorem{theorem}{Theorem}[section]

\newtheorem{corollary}[theorem]{Corollary}

\newtheorem{lemma}[theorem]{Lemma}
\newtheorem{proposition}[theorem]{Proposition}
\newtheorem{claim}[theorem]{Claim}

  {\end{tabular}\par\medskip}

\theoremstyle{definition}

\newtheorem{remark}[theorem]{Remark}

\newtheorem{definition}[theorem]{Definition}

\newtheorem{question}[theorem]{Open Question}
\DeclareMathOperator\Gr{Gr}

\newcommand{\GL}{\mathrm{GL}}
\newcommand{\F}{\mathbb{F}}

\newcommand{\bC}{\mathbb{C}}
\newcommand{\R}{\mathbb{R}}
\newcommand{\N}{\mathbb{N}}

\DeclareMathOperator{\trace}{Tr}

\newcommand{\rk}{\mathrm{rank}}

\DeclareMathOperator{\supp}{supp}

\newcommand{\bnorm}[1]{\left\lVert#1\right\rVert}
\newcommand{\norm}[1]{\lVert#1\rVert}
\newcommand{\ab}[1]{\left\lvert#1\right\rvert}

\newcommand{\M}{\mathrm{M}}

\newcommand{\U}{\mathrm{U}}

\newcommand{\tuple}[1]{\mathbf{#1}}

\newcommand{\V}{\mathcal{V}}

\newcommand{\vB}{\tuple{B}}

\newcommand{\vU}{\tuple{U}}

\newcommand{\colspan}{\operatorname{colspan}}

\newcommand{\E}{\mathrm{E}}

 \newcommand{\ann}{V^{\perp}}
 
\DeclareMathOperator{\rank}{rank}

\newcommand{\too}%
{\xrightarrow{\text{\raisebox{-3pt}{$\sim$}}\,}}

\newcommand\smat[1]{\left[\begin{smallmatrix}#1\end{smallmatrix}\right]}

\newlist{thmenum}{enumerate}{1}
\setlist[thmenum]{label=\rm(\arabic*), ref=\thetheorem(\arabic*)}
\crefalias{thmenumi}{theorem} 

\newlist{propenum}{enumerate}{1}
\setlist[propenum]{label=\rm(\arabic*), ref=\theproposition(\arabic*)}
\crefalias{propenumi}{proposition} 

\crefname{figure}{Figure}{Figures}

\allowdisplaybreaks[1]

\hypersetup{
    pdftitle={On linear-algebraic notions of expansion},
    pdfauthor={Yinan Li, Youming Qiao, Avi Wigderson, Yuval Wigderson, Chuanqi Zhang},
    bookmarksnumbered=true,     
    bookmarksopen=true,         
    bookmarksopenlevel=1,       
    colorlinks=true,            
    pdfstartview=Fit,           
    pdfpagemode=UseOutlines,    
    pdfpagelayout=OneColumn,
    pdfstartview=FitH
}

\title{
On linear-algebraic notions of expansion
}

\author{
Yinan Li\thanks{
	School of Mathematics and Statistics and Hubei Computational Science Key Laboratory, Wuhan University, Wuhan 430072, China
	({\tt 
Yinan.Li@whu.edu.cn}). Part of this work is done when Yinan was a designated assistant professor at Nagoya University and supported by MEXT Quantum Leap Flagship Program (MEXT Q-LEAP) Grant Number JPMXS0120319794.}
\and
Youming Qiao\thanks{Centre for Quantum Software and Information, 
 University of Technology Sydney, Australia ({\tt Youming.Qiao@uts.edu.au}). Research supported in part by Australian Research Council DP200100950.}
\and
Avi Wigderson\thanks{School of Mathematics, Institute for Advanced Study, 
Princeton, New Jersey 08540 ({\tt avi@ias.edu}). Research supported in part by NSF grant CCF-1900460.
}
\and
Yuval Wigderson\thanks{Department of Mathematics, Tel Aviv University, Israel ({\tt yuvalwig@tauex.tau.ac.il}). Research supported in part by NSF GRFP Grant DGE-1656518, ERC Consolidator Grant 863438, ERC Starting Grant 101044123 and NSF-BSF Grant 20196.}
\and
Chuanqi Zhang\thanks{Centre for Quantum Software and Information, 
 University of Technology Sydney, Australia ({\tt 
Chuanqi.Zhang@student.uts.edu.au}). Research supported by Australian Research Council DP200100950 and the Sydney Quantum Academy, Sydney, NSW, Australia.}
}

\date{\today}

\begin{document}

\maketitle
\begin{abstract}
A fundamental fact about bounded-degree graph expanders is that three notions of expansion---vertex expansion, edge expansion, and spectral expansion---are all equivalent. In this paper, we study to what extent such a statement is true for linear-algebraic notions of expansion.

There are two well-studied notions of linear-algebraic expansion, namely dimension expansion (defined in analogy to graph vertex expansion) and quantum expansion (defined in analogy to graph spectral expansion). 
Lubotzky and Zelmanov proved that the latter implies the former. We prove that the converse is false: there are dimension expanders which are not quantum expanders.

Moreover, this asymmetry is explained by the fact that there are two \emph{distinct} linear-algebraic analogues of graph edge expansion. The first of these is \emph{quantum edge expansion}, which was introduced by Hastings, and which he proved to be equivalent to quantum expansion. We introduce a new notion, termed \emph{dimension edge expansion}, which we prove is equivalent to dimension expansion and which is implied by quantum edge expansion. Thus, the separation above is implied by a finer one: dimension edge expansion is strictly weaker than quantum edge expansion. This new notion also leads to a new, more modular proof of the Lubotzky--Zelmanov result that quantum expanders are dimension expanders.

\vskip 1em 
{\centering \it Yinan Li would like to dedicate this paper to the memory of Keding Ma, a beloved grandfather who passed away at 88.}
\end{abstract}

\section{Introduction}
\subsection{Graph-theoretic and linear-algebraic notions of expansion}\label{subsec:overview}
Expansion is a fundamental graph-theoretic notion, with applications in and connections to combinatorics, geometry, group theory, number theory, probability, theoretical computer science, and many other fields. For an in-depth introduction to expanders and their applications, we refer the reader to the monograph \cite{HLW06}.

One of the reasons why expanders are so ubiquitous is that there are three different notions of expansion in graphs, which are all equivalent. These equivalences naturally yield connections between different perspectives on expansion, and allow expanders to be utilized and studied in many different contexts. We briefly recall the three notions of expansion.

Let $G=([n],E)$ be a $d$-regular graph. The \emph{edge expansion} of $G$, $h(G)$, is defined as\footnote{Some authors define edge expansion without the factor of $d$, but we use this normalization to match the definition of quantum edge expansion introduced by~\cite{PhysRevA.76.032315}.}
\begin{equation}\label{eq: graph edge expansion}
h(G)\coloneqq \min_{\substack{W \subseteq [n]\\1\leq|W|\leq\frac{n}{2}}}\frac{|\partial W|}{d|W|},
\end{equation}
where $\partial W\coloneqq \{\{i,j\}\in E:~i\in W, j\in [n]\setminus W\}$. The \emph{vertex expansion} of $G$, $\mu(G)$, is defined as
\begin{equation}\label{eq: graph vertex expansion}
\mu(G)\coloneqq \min_{\substack{W \subseteq [n]\\1\leq|W|\leq\frac{n}{2}}}\frac{|\partial_{out}(W)|}{|W|},
\end{equation}
where $\partial_{out}(W)\coloneqq \{j\in [n]\setminus W:~\exists\ i\in W, ~\text{s.t.}~\{i,j\}\in E\}$.
The \emph{spectral expansion}\footnote{This notion is also sometimes called the \emph{spectral gap} of $G$.} of $G$, $\lambda(G)$, is defined as
\begin{equation}
    \lambda(G) \coloneqq \text{the second-smallest eigenvalue of }L\label{eq:graph spectral expansion}
\end{equation}
where $L$ is the normalized Laplacian matrix of $G$, which is the matrix with
$L_{i,i}=1, L_{i,j}=-1/d$ if $\{i,j\}\in E$ and $L_{i,j}=0$ otherwise. 
Equivalently, $L=I_n-A$, where $I_n$ is the $n\times n$ identity matrix and $A$ is the normalized adjacency matrix of $G$, defined by
$A_{i,j}=1/d$ if $\{i,j\}\in E$ and $A_{i,j}=0$ otherwise. 

Note that for any $d$-regular graph $G$, all three quantities $h(G), \mu(G), \lambda(G)$ are non-negative. We say that a sequence of $d$-regular graphs $(G_n)_{n \in \N}$ is an \emph{edge expander} (resp.\ \emph{vertex expander}, \emph{spectral expander}) if the relevant parameter is \emph{uniformly} bounded away from zero for the whole family, namely if $\inf_n h(G_n)>0$ (resp.\ $\inf_n \mu(G_n)>0, \inf_n \lambda(G_n)>0$). As discussed above, these three notions are equivalent; the precise quantitative relationships between them are given in the following proposition. Note that in this proposition, the size of the graph is irrelevant and does not affect any of the bounds.
\begin{proposition}\label{prop: graph expansions relations}
Let $G$ be a $d$-regular graph. Then we have
\begin{propenum}
\item $\frac {\mu(G)}d\leq h(G)\leq \mu(G)$;\label{propit:vertex-edge}
\item $\frac{\lambda(G)}{2}\leq h(G)\leq \sqrt{2\lambda(G)}$.\label{propit:cheeger}
\end{propenum}
\end{proposition}
The proof of \cref{propit:vertex-edge} is straightforward. \cref{propit:cheeger} is a discrete analogue of the celebrated \emph{Cheeger's inequality} \cite{Cheeger70}, proved by Dodziuk \cite{Dod84}, and independently by Alon--Milman \cite{AM85} and Alon \cite{Alo86}. Note that we think of the degree $d$ as a constant, so that we only lose a constant factor when moving between the notions of vertex and edge expansion, and only lose a quadratic factor when moving between these and the notion of spectral expansion.

In this paper, we are interested in studying linear-algebraic notions of expansion and their relationships. There are several well-studied notions of linear-algebraic expansion, including dimension expanders (introduced by Barak, Impagliazzo, Shpilka, and Wigderson \cite{BISW04}), quantum expanders (introduced independently by Hastings \cite{PhysRevA.76.032315} and by Ben-Aroya and Ta-Shma \cite{BT07}), and quantum edge expanders (introduced by Hastings \cite{PhysRevA.76.032315}). We will not discuss here the motivations for these definitions (besides being very natural extensions of the related graph-theoretic parameters), but note that they have led to much further exploration. For dimension expansion see e.g.\ \cite{DS09,BOURGAIN2009357,Bourgain2013,DW10,FG15,LUBOTZKY2008730} and for quantum expansion see e.g.\ \cite{Harrow08,BST10,10.5555/2011781.2011790,zigzag,FM20,KLR21}. Many of these papers and others deal with the important problem of explicitly constructing quantum and dimension expanders; some make use of connections between different notions of expansion, and have led to the introduction of new notions of expansion\footnote{e.g.\ monotone expanders \cite{DS09,Bourgain2013,DW10}.}. The linear-algebraic notions of expansion will be formally defined momentarily, but first we wish to make some high-level remarks about them. 

Dimension expansion is defined in natural analogy to the graph-theoretic definition of vertex expansion, quantum expansion is defined in natural analogy to spectral expansion, and quantum edge expansion is defined in natural analogy to edge expansion. Because of these analogies, it is natural to wonder whether the three notions are equivalent. Hastings \cite{PhysRevA.76.032315} proved an analogue of \cref{propit:cheeger}, showing that quantum expansion and quantum edge expansion are equivalent. Additionally, it is implicit in work of Lubotzky and Zelmanov \cite{LUBOTZKY2008730} that (under mild assumptions) quantum expansion implies dimension expansion. However, no reverse implication was known, nor any analogue of \cref{propit:vertex-edge} relating dimension expansion and quantum edge expansion.

Our first result is that such statements, showing that dimension expansion implies quantum expansion or quantum edge expansion, are \emph{false}. Indeed, we show the existence of dimension expanders that are arbitrarily poor quantum expanders (and thus arbitrarily poor quantum edge expanders). Moreover, we are able to explain ``why'' no such equivalence holds: it is because there is a ``missing'' fourth notion of linear-algebraic expansion, which we term \emph{dimension edge expansion}. This is yet another natural linear-algebraic analogue of edge expansion, which had not been previously defined. For this notion, it is straightforward to show an analogue of \cref{propit:vertex-edge}, proving that dimension expansion and dimension edge expansion are equivalent. 
Additionally, we prove that quantum edge expansion implies dimension edge expansion.

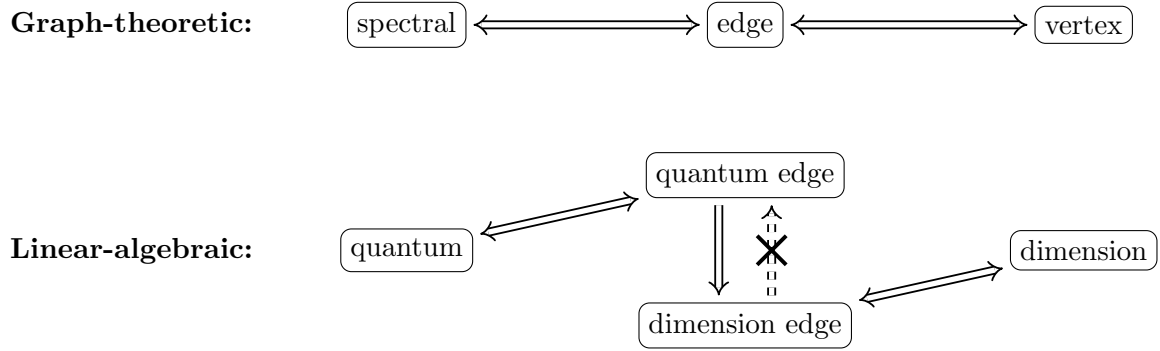
\begin{figure}[t]
\begin{center}
\tikzset{cross/.style={cross out, draw, 
         minimum size=2*(#1-\pgflinewidth), 
         inner sep=0pt, outer sep=0pt}}
    \begin{tikzpicture}[double distance=2pt, outer sep=1mm, exp/.style={draw, rectangle, rounded corners}, eq/.style={{Implies[]}-{Implies[]}, double, line width=.7pt}, ar/.style={-{Implies[]}, double, line width=.7pt}]
        \node[exp] at (0,0) (edge) {edge};
        \node[exp] (spectral) at (-4.5,0) {spectral};
        \node[exp] at (4.5,0) (vertex) {vertex};
        \node[anchor=west] at (-10,0) {\bf Graph-theoretic:};
        \draw[eq] (spectral) -- (edge) ;
        \draw[eq] (vertex) -- (edge) ;

        \node[anchor=west] at (-10,-3) {\bf Linear-algebraic:};
        \node[exp] at (-4.5,-3) (quant) {quantum};
        \node[exp] at (0,-2) (quantedge) {quantum edge};
        \node[exp] at (0,-4) (dimedge) {dimension edge};
        \node[exp] at (4.5,-3) (dim) {dimension};
        \draw[eq] (quant) -- (quantedge);
        \draw[eq] (dim) -- (dimedge);
        \draw[ar] ([xshift=-10pt]quantedge.south) -- ([xshift=-10pt]dimedge.north);
        \draw[ar, dashed] ([xshift=10pt]dimedge.north) -- ([xshift=10pt]quantedge.south) node[pos=.5, cross=7pt, black, solid, line width=1.5pt] {};
    \end{tikzpicture}
\caption{A schematic depiction of the relationships between different notions of expansion. }
\label{fig:expansion-notions}
\end{center}
\end{figure}

To understand what all these implications mean,
consider \cref{fig:expansion-notions} above, which clarifies the conceptual value of the new definition. We stress, as the figure suggests, that both linear-algebraic notions of edge expansion specialize to the same graph-theoretic one. Additionally, the figure shows the two equivalences discussed above, namely that both ``quantum'' notions are equivalent, and both ``dimension'' notions are equivalent. Moreover, it depicts the fact that quantum edge expansion implies dimension edge expansion. This connection yields a new proof of the Lubotzky--Zelmanov result discussed above; moreover, our new proof is more modular, and gives a result that is both qualitatively and quantitatively stronger.

Finally, as depicted in the figure, our negative result shows that there is no reverse implication, as there exist dimension expanders which are not quantum expanders. Said differently, in the linear-algebraic setting, the ``quantum'' notions of expansion are strictly stronger than the ``dimension'' notions. We stress again the surprising consequence: although both linear-algebraic notions of edge expansion generalize the \emph{same} graph-theoretic notion, they are not equivalent.



We now turn to the formal definitions of these linear-algebraic notions.
Once the definitions are in place, we can state our main theorems.

\subsection{Definitions of linear-algebraic expansion}

Throughout, our main object of study will a \emph{matrix tuple} $\vB=(B_1,\dots,B_d) \in \M(n,\F)^d$, where $\F$ is some field and $\M(n,\F)$ denotes the space of $n\times n$ matrices over $\F$. The above notions of linear-algebraic expansion will all be properties of such matrix tuples $\vB$. The analogies to graphs will be apparent when considering $\vB$ as a tuple of \emph{permutation} matrices, which naturally define a $d$-regular graph\footnote{For it to define an undirected graph, it must be a symmetric set of permutations, but this is essentially without loss of generality. Additionally, it is well-known \cite{Gross77} that any $d$-regular graph can be decomposed as a union of $d$ permutations (at least when $d$ is even), so we may always view a $d$-regular graph as a tuple of permutation matrices.}.

\subsubsection{Quantum expansion and quantum edge expansion}
When working with quantum expanders and quantum edge expanders, we work with the field $\F=\bC$. Additionally, rather than working with arbitrary matrix tuples, we work with \emph{doubly stochastic matrix tuples}, which are those tuples $\vB=(B_1,\dots,B_d) \in \M(n,\bC)^d$ with $\sum_{i=1}^d B_i B_i^* = \sum_{i=1}^d B_i^* B_i=dI_n$, where $I_n$ is the $n\times n$ identity matrix. An important special case is that of \emph{unitary matrix tuples}, where each $B_i$ is a unitary matrix.

\begin{definition}
    Given a doubly stochastic matrix tuple $\vB=(B_1,\dots,B_d) \in \M(n,\bC)^d$, the \emph{associated quantum operator} is the linear map $\Phi_\vB:\M(n,\bC) \to \M(n,\bC)$ defined by
    \[
        \Phi_\vB (X) = \frac 1d \sum_{i=1}^d B_i X B_i^*.
    \]
\end{definition}
The quantum operator $\Phi_\vB$ should be thought of as an analogue of the normalized adjacency matrix $A$ of a graph $G$. Indeed, it is straightforward to check that if each $B_i$ is a permutation matrix, and if $X$ is a diagonal matrix, then $\Phi_\vB(X)$ is also diagonal, and the diagonal entries are precisely the entries of $Ax$, where $x$ is the vector of diagonal entries of $X$. Similarly, it is easy to verify that the largest eigenvalue of $\Phi_\vB$ is $1$, with eigenvector $I_n$. Continuing the analogy, we  define the \emph{Laplacian operator} by $\Lambda_\vB \coloneqq \mathcal{I}- \Phi_\vB$, where $\mathcal{I}$ is the identity map $\M(n,\bC) \to \M(n,\bC)$.

Based on this analogy between quantum operators and adjacency matrices, the following definition\footnote{We remark that often, quantum expansion is defined for a quantum operator, rather than for a tuple of matrices, where a \emph{quantum operator} is defined abstractly as a linear map $\M(n,\bC) \to \M(n,\bC)$ satisfying certain properties. However, it is well-known \cite[Theorem 8.2]{NC00} that any quantum operator arises in an essentially unique way from a doubly stochastic matrix tuple, so the two perspectives are equivalent.} from \cite{PhysRevA.76.032315,BT07} is a natural analogue of the definition of spectral expansion of a graph.
\begin{definition}
    Given a doubly stochastic matrix tuple $\vB=(B_1,\dots,B_d) \in \M(n,\bC)^d$, its \emph{quantum expansion}, $\lambda(\vB)$, is defined to be the second-smallest singular value\footnote{In general, $\Phi_\vB$ (and thus also $\Lambda_\vB$) may not be a self-adjoint operator on $\M(n,\bC)$, and hence its eigenvalues may not be real; this is why we restrict our attention to singular values. In many special cases (e.g.\ if all the matrices $B_i$ are Hermitian, or if the tuple $\vB$ is a symmetric set of permutation matrices), $\Phi_\vB$ and $\Lambda_\vB$ are self-adjoint, and then we may equivalently define $\lambda(\vB)$ as the second-smallest eigenvalue of $\Lambda_\vB$.} of $\Lambda_\vB$.

    For fixed $d$ and $\lambda>0$, we say that a family of doubly stochastic matrix tuples $\{\vB_n=(B_1, \dots, B_d)\in \M(n, \bC)^d\mid n\in \N\}$ is an \emph{$(n,d,\lambda)$-quantum expander} if $\lambda(\vB_n)\geq \lambda$ for all $n\in\N$. 
\end{definition}


Similarly, the following definition from \cite{PhysRevA.76.032315} is a linear-algebraic analogue of the edge expansion of a graph.

\begin{definition}
    Given a doubly stochastic matrix tuple $\vB =(B_1,\dots,B_d) \in \M(n,\bC)^d$, its \emph{quantum edge expansion} is defined as
    \begin{equation}\label{eq: quantum edge expansion}
    h_Q(\vB)\coloneqq \min_{\substack{V\leq \bC^n\\ 1\leq\dim(V)\leq \frac n2}}\frac{\langle 
    I_n-P_V,\Phi_\vB(P_V)\rangle}{\dim(V)},
    \end{equation}
    where $P_V$ is the orthogonal projection onto the subspace $V\leq \bC^n$, and where $\langle \cdot,\cdot\rangle$ denotes the standard inner product on $\M(n,\bC)$.
    
    For fixed $d$ and $h>0$, we say that a family of doubly stochastic matrix tuples $\{\vB_n=(B_1, \dots, B_d)\in \M(n, \bC)^d\mid n\in \N\}$ is an \emph{$(n,d,h)$-quantum edge expander} if $h_Q(\vB_n)\geq h$ for all $n\in\N$. 
\end{definition}
Suppose again that each $B_i$ is a permutation matrix. Let $W \subseteq [n]$, and suppose that $V = \langle e_j\rangle_{j \in W}$ is a \emph{coordinate subspace}, spanned by a subset of the standard basis $\{e_1,\dots,e_n\}$ of $\bC^n$. Then $P_V$ is simply a diagonal matrix with a $1$ in positions indexed by $W$ and a $0$ elsewhere. Similarly, as discussed above, $\Phi_\vB(P_V)$ is another diagonal matrix, whose diagonal entries are precisely the entries of $A\mathbf 1_W$, where $A$ is the normalized adjacency matrix of $G$ and $\mathbf 1_W$ is the indicator vector of $W$. Therefore, $\langle I_n-P_V, \Phi_\vB(P_V) \rangle$ is equal to $\mathbf 1_{[n] \setminus W} A \mathbf 1_W$, which in turn equals $\frac 1 d \ab{\partial W}$. Thus, when $\vB$ comprises permutation matrices and when we restrict the minimum to coordinate subspaces $V \leq \bC^n$, the definition of quantum edge expansion precisely recovers the definition of edge expansion of a graph.

\subsubsection{Dimension expanders and dimension edge expanders}\label{subsec:dim_exp}
Dimension expansion and dimension edge expansion are well-defined over any field, but for simplicity, we continue working with $\F=\bC$ for the moment.

Given a tuple of matrices $\vB=(B_1, \dots, B_d)\in \M(n, 
\bC)^d$, the image of $V\leq\bC^n$ under $\vB$ is 
$\vB(V)\coloneqq \langle \cup_{i\in[d]}B_i(V)\rangle$, where $B_i(V) \coloneqq \{B_iv:v \in V\}$ and $\langle\cdot\rangle$ denotes 
linear span over $\bC$. 
\begin{definition}
The \emph{dimension expansion} of a matrix tuple $\vB=(B_1,\dots,B_d)\in \M(n,\bC)^d$ is defined as
\begin{equation}\label{eq: dimension expansion}
    \mu(\vB)\coloneqq \min_{\substack{V\leq \bC^n \\1\leq\dim(V)\leq \frac n2}}\frac{\dim(V+\vB(V))-\dim(V)}{\dim(V)}.
\end{equation}
For fixed $d$ and $\mu >0$, we say that a family of matrix tuples $\{\vB_n=(B_1, \dots, B_d)\in \M(n, \bC)^d\mid n\in \N\}$ is an \emph{$(n,d,\mu)$-dimension expander} if $\mu(\vB_n)\geq \mu$ for all $n\in\N$. 
\end{definition}
To see the analogy to vertex expansion in graphs, suppose that each $B_i$ is a permutation matrix, let $W \subseteq [n]$, and suppose that $V = \langle e_j\rangle_{j \in W}$ is a {coordinate subspace}. Then $B_i(V)$ is another coordinate subspace, spanned by the images of $W$ under the permutation $B_i$. Therefore, $\dim(V+\vB(V))-\dim(V)$ precisely counts how many elements in the complement of $W$ are in the image of some $B_i$. In other words, if we restrict the definition of dimension expansion to tuples of permutation matrices, and restrict the minimum to {coordinate subspaces} $V \leq \bC^n$, we precisely recover the definition of vertex expansion in graphs.


The same perspective motivates our new definition of dimension edge expanders. To define them, we first need to define the restriction of a matrix to a subspace.
Let $V \leq \bC^n$ be a subspace of dimension $r$, and let $T \in \M(n\times r,\bC)$ be a matrix whose columns form an orthonormal basis of $V$. Similarly, let $R \in \M( (n-r) \times n, \bC)$ be a matrix whose rows form an orthonormal basis of $V^\perp$,
the orthogonal complement of $V$. Given $B \in \M(n,\bC)$, its restriction to $(V^\perp, V)$ is defined by $B|_{V^\perp, V}\coloneqq RBT$.
Note that while $T$ and $R$ are not unique, a different choice of $T$ and $R$ would give rise to a matrix that is equivalent to $B|_{V^\perp, V}$ up to multiplication by unitary matrices (and thus in particular has the same rank).

With this setup, we can define dimension edge expansion.

\begin{definition}\label{def:dim_edge_exp}
For a matrix tuple $\vB=(B_1, \dots, B_d)\in\M(n, \bC)^d$, the \emph{dimension edge expansion} of $\vB$ is defined as
\begin{equation}\label{eq: dim edge expansion}
    h_D(\vB)=\min_{\substack{V \leq \bC^n\\
            1\leq\dim(V)\leq \frac n2}}\frac{\sum_{i=1}^d\rk(B_i|_{V^\perp, V})}{d\cdot\dim(V)}.
\end{equation}
For fixed $d$ and $h>0$, we say that a family of matrix tuples $\{\vB_n=(B_1, \dots, B_d)\in \M(n, \bC)^d\mid n\in \N\}$ is an \emph{$(n,d,h)$-dimension edge expander} if $h_D(\vB_n)\geq h$ for all $n\in\N$. 
\end{definition}
Again suppose that $V=\langle e_i\rangle_{i \in W}$ is a coordinate subspace corresponding to some $W \subseteq [n]$.
Then $B_i|_{V^\perp,V}$ is simply the submatrix of $B_i$ with columns indexed by $W$ and rows indexed by $[n] \setminus W$. In case $B_i$ is a permutation matrix, the rank of this submatrix is precisely the number of non-zero entries in it, which is equal to the number of elements of $W$ mapped to $[n] \setminus W$ by the permutation $B_i$. In other words, if we restrict our attention to permutation matrices and to coordinate subspaces, we again recover the definition of edge expansion of graphs.

\begin{remark}\label{rmk:def-annihilator}
	To extend the definition of dimension edge expansion to any field $\F$, we denote by $V^\perp$ the \emph{annihilator} of $V\leq \F^n$ with respect to the standard dot product, namely $V^{\perp}:=\left\{u \in\mathbb{F}^n \mid u^t v=0\; \forall v \in V\right\}$. 
    Then let $T\in \M(n\times r, \F)$ be a matrix whose columns form a basis of $V$, and $R\in \M((n-r)\times n, \F)$ whose rows form a basis of $V^\perp$. Given $B \in \M(n,\F)$, its restriction to $(V^\perp, V)$ is defined by $B|_{V^\perp, V}\coloneqq RBT$. Dimension edge expansion is now defined identically to \cref{def:dim_edge_exp}.
\end{remark}


\subsection{Main results}
As discussed in \cref{subsec:overview}, our results demonstrate that in the linear-algebraic setting, the relationship between the various notions of expansion is substantially subtler than it is in the graph-theoretic setting. Our first main result is negative: it says that in general, a matrix tuple can be a good dimension expander while having arbitrarily poor quantum expansion. 
\begin{theorem}\label{thm:main-negative}
    A dimension expander may be an arbitrarily poor quantum expander. More precisely,
    there are constants $\mu>0$ and $d \in \N$ so that for all $\varepsilon>0$ and all sufficiently large $n$, there exists a unitary matrix tuple $\vB =(B_1,\dots,B_d) \in \M(n,\bC)^d$ such that $\mu(\vB) \geq \mu$ but $\lambda(\vB) < \varepsilon$.
\end{theorem}
In fact, we prove something stronger: given any unitary matrix tuple $\vB_0 \in \M(n,\bC)^d$, we can find another unitary matrix tuple $\vB \in \M(n,\bC)^d$ so that $\mu(\vB) \geq \mu(\vB_0)/d$ and $\lambda(\vB) < \varepsilon$. Moreover $\vB$ is obtained from $\vB_0$ in a very simple way: each matrix in $\vB$ is simply a very small fractional power of the corresponding matrix in $\vB_0$. Perhaps surprisingly, our proof uses a number of compactness arguments, and thus gives no quantitative information on how small this fractional power must be.

In the other direction, we have a number of relationships between the various notions of linear-algebraic expansion, in analogy to \cref{prop: graph expansions relations}.
\begin{theorem}\label{thm:main-relationships}
    Let $\vB=(B_1,\dots,B_d) \in \M(n,\bC)^d$ be a doubly stochastic matrix tuple. Then we have
    \begin{thmenum}
        \item $\frac {\mu(\vB)}d \leq h_D(\vB) \leq \mu(\vB)$;\label{thmit:dim}
        \item $\frac{\lambda(\vB)}{2} \leq h_Q(\vB) \leq \sqrt{2\lambda(\vB)}$;\label{thmit:quant-cheeger}
        \item $h_Q(\vB) \leq d\cdot h_D(\vB)$.\label{thmit:edge-comparison}
    \end{thmenum}
    In case $\vB$ is a unitary matrix tuple, we may replace {\rm (3)} by the stronger bound
    \begin{thmenum}[start=4]
        \item $h_Q(\vB) \leq h_D(\vB)$.\label{thmit:unitary-edge-comparison}
    \end{thmenum}
\end{theorem}
In fact, \cref{thmit:dim} holds for any matrix tuple over any field; the assumptions that we are working over $\bC$ and have a doubly stochastic matrix tuple are only necessary so that $\lambda(\vB)$ and $h_Q(\vB)$ are well-defined. As mentioned above, \cref{thmit:quant-cheeger} is not new, and was proved by Hastings \cite[Appendix A]{PhysRevA.76.032315}. So we do not prove \cref{thmit:quant-cheeger} in this paper, but we include it in the statement of \cref{thm:main-relationships} in order to make the analogy to \cref{prop: graph expansions relations} as transparent as possible.

Speaking of \cref{prop: graph expansions relations}, the key thing to stress about the difference between the graph-theoretic and linear-algebraic settings is that the \emph{single} graph-theoretic notion of edge expansion, which appears in both \cref{propit:cheeger,propit:vertex-edge}, corresponds to \emph{two distinct} linear-algebraic notions, namely dimension edge expansion and quantum edge expansion. That these two quantities are not equivalent is implied by \cref{thm:main-negative}, which gives a separation between $\mu$ and $\lambda$. However, the fact that there is an inequality relating them---\cref{thmit:edge-comparison}---shows that quantum expansion implies dimension expansion, as stated in the next corollary.

\begin{corollary}\label{cor:quantum-dimension}
    Every quantum expander is a dimension expander. More precisely, if $\vB \in \M(n,\bC)^d$ is a doubly stochastic matrix tuple, then
    \[
        \mu(\vB) \geq  \frac{\lambda(\vB)}{2d}.
    \]
    In case $\vB$ is a unitary matrix tuple, we have the stronger bound
    \[
        \mu(\vB) \geq  \frac{\lambda(\vB)}{2}.
    \]
\end{corollary}
\begin{proof}
    We have that
    \[
        \mu(\vB) \geq  h_D(\vB) \geq  \frac{h_Q(\vB)}{d} \geq  \frac{\lambda(\vB)}{2d},
    \]
    where the three inequalities follow from the first three parts of \cref{thm:main-relationships}. In case $\vB$ is unitary, we may replace the second inequality above by $h_D(\vB) \geq  h_Q(\vB)$, thanks to \cref{thmit:unitary-edge-comparison}.
\end{proof}

As discussed above, a similar result was proved by Lubotzky and Zelmanov \cite{LUBOTZKY2008730}\footnote{Lubotzky and Zelmanov were unaware of the notion of quantum expansion, and thus did not state this result in this language. Harrow \cite{Harrow08} was the first to observe that their approach passes through quantum expansion as an intermediate step.}.
\begin{theorem}[{\cite[Proposition 2.1]{LUBOTZKY2008730}}]\label{thm:LZ}
    Let $\vB=(B_1,\dots,B_d)\in\mathrm{M}(n,\bC)^d$ be a unitary matrix tuple. Then
    $$\mu(\vB) \geq \frac{\lambda(\vB)}{6}.$$
\end{theorem}
Lubotzky and Zelmanov only stated this result for matrix tuples arising from irreducible unitary representations of finite groups, but their proof actually works in the generality of arbitrary unitary matrix tuples. 

Note that our result (\cref{cor:quantum-dimension}) is both quantitatively stronger (in that we obtain a better constant factor) and qualitatively stronger (as we obtain a result for all doubly stochastic matrix tuples, not only unitary tuples). Additionally, we believe that our proof is conceptually simpler than that of \cite{LUBOTZKY2008730}, since the implication from quantum to dimension expansion naturally breaks into three simpler implications (quantum implies quantum edge, which implies dimension edge, which implies dimension).

Lubotzky and Zelmanov \cite{LUBOTZKY2008730} proved \cref{thm:LZ} in order to explicitly construct dimension expanders over $\bC$. Indeed, this motivation is a natural reason to study the relationships between linear-algebraic notions of expansion: \cref{cor:quantum-dimension} implies that any construction of a quantum expander also yields a construction of a dimension expander, whereas \cref{thm:main-negative} shows that the converse does not hold.
We remark that explicit constructions of quantum and dimension expanders is an important and highly active area of research, which we do not discuss further, except to stress that this is a very natural reason to study connections between different notions of expansion.

\subsection{Connections between graphs and matrix spaces}
There is a natural way of associating a matrix tuple to a graph. Namely, given a $d$-regular graph $G=([n],E)$, we define the \emph{graphical matrix tuple} to be the matrix tuple\footnote{Note that we scale $\E_{i,j}$ by a factor of $\sqrt n$ in order to ensure that $\vB_G$ is a doubly stochastic matrix tuple, as $\sum_{\{i,j\}\in E} (\sqrt n \E_{i,j})(\sqrt n \E_{i,j})^*=\sum_{\{i,j\}\in E} (\sqrt n \E_{i,j})^*(\sqrt n \E_{i,j})=dnI_n$.} $\vB_G = (\sqrt n\E_{i,j}:\{i,j\} \in E)$, where $\E_{i,j}$ is the elementary matrix with a $1$ in position $(i,j)$ and zeros in all other entries. Note that $\vB_G$ is a tuple of $2\ab E=dn$ matrices, and thus is not particularly natural from the perspective of expansion: we are usually interested in families of matrix tuples where the length of the tuple stays constant as the dimension of the matrices grows. Here, even if the degree $d$ of $G$ is fixed, the number of matrices in $\vB_G$ tends to infinity with $n$.

Nonetheless, the construction of $\vB_G$ is very natural from other perspectives. Indeed, the matrices in $\vB_G$ form the standard basis for the \emph{graphical matrix space} associated to $G$. In \cite{moreconnections}, we proved that many important graph-theoretic properties of $G$ are equivalent to linear-algebraic properties of $\vB_G$. In particular, we proved a number of results that we term \emph{inherited correspondences}, which say that the value of a certain optimum associated to $G$ is equal to a related optimum associated to $\vB_G$, even though the feasible region in the latter optimum is generally much larger; for example, an optimum over all subgraphs of $G$ may equal an optimum over all subspaces of $\langle \vB_G\rangle$, even though there are many more subspaces than subgraphs.

As it turns out, the various expansion parameters give us a number of other results of this type. The first of these is due to Bannink, Bri\"{e}t, Labib, and Maassen \cite{quasirandom}, who proved that the quantum expansion of $\vB_G$ equals the spectral expansion of $G$.
\begin{theorem}[{\cite[Proposition 3.7]{quasirandom}}]\label{prop: spectral 
expansion generalization}
For every $d$-regular graph $G$, $\lambda(G)=\lambda(\vB_G)$.
\end{theorem}
We prove analogous results for dimension expansion and dimension edge expansion.
\begin{theorem}\label{thm:graphical-tuple-preserves}
    For every $d$-regular graph $G$, $\mu(G) = \mu(\vB_G)$ and $h(G) = h_D(\vB_G)$.
\end{theorem}

In contrast, no such equality holds for quantum edge expansion. In fact, it fails already for the simplest possible graph, consisting of two vertices and a single edge.
\begin{proposition}\label{prop:hq-k2}
    $h(K_2) =1$, but $h_Q(\vB_{K_2})\leq\frac 12$.
\end{proposition}
Using the same proof technique, one can check that $h_Q(\vB_{K_n}) < h(K_n)$ for all $n \geq 2$. The fact that $h(G)$ and $h_Q(\vB_G)$ are different in general may shed further light on our main results in \cref{thm:main-negative,thm:main-relationships}, that dimension edge expansion and quantum edge expansion are different. However, we stress that these simple examples have no direct bearing on \cref{thm:main-negative,thm:main-relationships}, as $\vB_G$ is not a matrix tuple of constant length, which is the setting in which \cref{thm:main-negative,thm:main-relationships} are interesting.

\begin{remark}
    We include the $\sqrt n$ normalization in the definition of $\vB_G$ in order to obtain a doubly stochastic tuple. When working with a non-$d$-regular graph or over a field other than $\bC$, one should omit this normalization; with this definition, \cref{thm:graphical-tuple-preserves} holds for any graph $G$ and over any field.
\end{remark}

\subsection{Notation and basic definitions}\label{sec:prel}

For $n\in \N$, $[n]\coloneqq \{1, \dots, n\}$. 
Throughout, we work with finite-dimensional vector spaces over a field $\F$. 
The elements of $\F^n$ are length-$n$ \emph{column} vectors over 
$\F$. 
The linear space of 
$n'\times n$ matrices over $\F$ is denoted $\M(n'\times n,\F)$. For simplicity we shall write 
$\M(n,\F)$ for 
$\M(n\times n,\F)$. We use $\E_{i, j}\in \M(n'\times n, \F)$ to denote 
the $n'\times n$ elementary matrix where the $(i, j)$th entry is $1$, and other entries are $0$. 

We use $\langle x, y\rangle$ to denote inner products. In particular, for $x=(x_1, \dots, x_n)^t,y=(y_1, \dots, y_n)^t\in\bC^n$, their inner product is
$
\langle x,y\rangle\coloneqq \sum_{i=1}^n \overline{x_i}y_i.
$
For every $p\in[1,\infty)$, the $L_p$-norm of $x\in\bC^n$ 
is defined as 
\[
\norm{x}_{L_p}\coloneqq \left(\sum^n_{i=1}|x_i|^p\right)^{\frac{1}{p}}.
\]
Define $\norm{x}_{L_\infty}\coloneqq \max_i|x_i|$. 

We use $\trace(\cdot)$ to denote the usual trace function on $\M(n,\F)$. For $X,Y\in \M(n,\bC)$, their inner product is $\langle X,Y\rangle\coloneqq \trace(X^*Y)$.
For every $p\in[1,\infty)$, the Schatten-$p$ norm of $X\in \M(n,\bC)$ is defined as
\[
\norm{X}_{S_p}\coloneqq \left(\trace\left[(X^*X)^{\frac{p}{2}}\right]\right)^{\frac{1}{p}}.
\]
Define $\norm{X}_{S_\infty}\coloneqq \sup\{|\langle 
x,Xy\rangle|:~\norm{x}_{L_2},\norm{y}_{L_2}\leq 1\}$, which is the operator norm 
of $X$. 

\subsection{Organization}
In \cref{sec:dim-not-quantum}, we prove \cref{thm:main-negative}, showing that a small fractional power of a dimension expander yields another dimension expander which is an arbitrarily poor quantum expander. In \cref{sec:quantum-to-dim}, we prove \cref{thm:main-relationships}, showing that dimension and dimension edge expansion are equivalent, and that quantum edge expansion implies dimension edge expansion. In \cref{sec:connections}, we prove \cref{thm:graphical-tuple-preserves,prop:hq-k2}, relating graph expansion and linear-algebraic expansion of the associated graphical matrix tuple. We conclude in \cref{sec:conclusion} with some open problems and concluding remarks.

\section{Dimension expanders that are not quantum expanders}\label{sec:dim-not-quantum}
In this section, we prove \cref{thm:main-negative}. Our main technical result is the following, which says that given an arbitrary unitary matrix tuple, any sufficiently small fractional power of it has equal or greater dimension expansion. As it is easy to show that a small fractional power of any unitary matrix tuple has poor quantum expansion, it becomes straightforward to deduce \cref{thm:main-negative}, as we show below.

Recall that given a unitary matrix $U$, there exists a Hermitian matrix $H$ such that $U = e^{iH}$. Using this, we can define arbitrary powers of $U$: for any $s>0$, we define $U^s \coloneqq e^{isH}$. Note that this notion depends on the choice of $H$, and a different choice of $H$ may yield a different outcome, but this will not matter for us. So for any unitary matrix $U$, we fix a choice of Hermitian $H$ such that $U=e^{iH}$, and then define powers of $U$ with respect to this choice.
\begin{theorem}\label{thm:exp-keeps-dim-exp}
    Let $\vU=(U_1,\dots,U_d)\in \mathrm{U}(n)^d$ be a unitary tuple of matrices. There exists some $s_0>0$ such that for all $s \in (0,s_0)$, the unitary matrix tuple $\vU^s=(U_1^s,\dots,U_d^s)\in\mathrm{U}(n)^d$ satisfies $\mu(\vU^s) \geq \mu(\vU)/d$.
\end{theorem}


Assuming \cref{thm:exp-keeps-dim-exp}, we can prove \cref{thm:main-negative}.

\begin{proof}[Proof of \cref{thm:main-negative}]

    We recall the variational definition of quantum expansion: for any doubly stochastic matrix tuple $\vB$, we have that
    \begin{equation}\label{eq:variational-lambda}
        \lambda(\vB) = 1- \max_{\substack{0 \neq X \in \M(n,\bC)\\\trace(X)=0}} \frac{\norm{\Phi_{\vB}(X)}_{S_2}}{\norm X_{S_2}}.
    \end{equation}
    Indeed, we first note that the second-smallest singular value of $\Lambda_{\vB}$ equals one minus the second-largest singular value of $\Phi_{\vB}$. Since the largest eigenvalue of $\Phi_{\vB}$ is $1$, with $I_n$ as both a left and a right eigenvector, we know that the second-largest singular value of $\Phi_{\vB}$ is simply the operator norm of $\Phi_{\vB}$ when restricted to the orthogonal complement of $I_n$. The orthogonal complement of $I_n$ is the set of matrices $X$ with $\trace(X)=0$, and thus the maximum in \eqref{eq:variational-lambda} precisely defines the second-largest singular value of $\Phi_{\vB}$. In particular, by the triangle inequality, we see that
    \begin{equation}
        \lambda(\vB) = \max_{\substack{0 \neq X \in \M(n,\bC)\\\trace(X)=0}} \frac{\norm X_{S_2} - \norm{\Phi_\vB(X)}_{S_2}}{\norm X_{S_2}} \leq \max_{\substack{0 \neq X \in \M(n,\bC)\\\trace(X)=0}} \frac{\norm{X-\Phi_\vB(X)}_{S_2}}{\norm X_{S_2}}. \label{eq:lambda-ub}
    \end{equation}

	It is well-known that for constant $d \in \N$ and $\mu>0$, there exist unitary $(n,d,d\mu)$-dimension expanders for all sufficiently large $n$ (e.g.\ by taking $100$ random $n \times n$ unitary matrices). So fix some $\vU=(U_1,\dots,U_d)\in \mathrm{U}(n)^d$ with $\mu(\vU) \geq d\mu$. By \cref{thm:exp-keeps-dim-exp}, we know that $\vU^s$ is still an $(n,d,\mu)$-dimension expander for any sufficiently small $s$. Given $\varepsilon>0$, we claim that $\lambda(\vU^s)< \varepsilon$ for sufficiently small $s$.

    Note that as $s$ tends to $0$, $\vU^s$ approaches the identity tuple $(I_n,\dots,I_n)$. More precisely, for any $\varepsilon>0$ and any $i\in[d]$, there exists a sufficiently small $s$ such that 
    \begin{equation}\label{eq:close-to-id}
         \norm{I_n - U_i^s}_{S_2}<\frac \varepsilon 2 \qquad \text{ and } \qquad \norm{I_n - (U_i^s)^*}_{S_2} < \frac \varepsilon 2.
    \end{equation}
    Note that the two conditions are actually equivalent, as the Schatten norm is invariant under taking conjugate transpose.
    Now fix $s$ sufficiently small so that \eqref{eq:close-to-id} holds for all $i \in [d]$, and so that $\mu(\vU^s) \geq \mu(\vU)/d \geq \mu$. It remains to prove that for this choice of $s$, we have $\lambda(\vU^s)<\varepsilon$. 

    Fix some $0 \neq X \in \M(n,\bC)$ with $\trace(X) \neq 0$. By our choice of $s$, we have that
    \begingroup
    \allowdisplaybreaks
    \begin{align}
    \frac{\|X-\Phi_{\vU^s}(X)\|_{S_2}}{\norm{X}_{S_2}}
    &=\frac{\norm{\sum_{i=1}^{d}(X-U_i^s X(U_i^s)^*)}_{S_2}}{d\norm{X}_{S_2}} \notag\\
    &\leq\sum_{i=1}^{d}\frac{\norm{X(I_n-(U_i^s)^*)+(I_n-U_i^s)X(U_i^s)^*}_{S_2}}{d\norm{X}_{S_2}} \label{eq:triangle-1}\\
    &\leq\sum_{i=1}^{d}\frac{\norm{X(I_n-(U_i^s)^*)}_{S_2}+\norm{(I_n-U_i^s)X(U_i^s)^*}_{S_2}}{d\norm{X}_{S_2}}\label{eq:triangle-2}\\
    &=\sum_{i=1}^{d}\frac{\norm{X(I_n-(U_i^s)^*)}_{S_2} + \norm{(I_n-U_i^s)X}_{S_2}}{d\norm{X}_{S_2}}\label{eq:unitary-invar}\\
    &\leq \sum_{i=1}^d \frac{\norm{I_n-(U_i^s)^*}_{S_2} + \norm{I_n-U_i^s}_{S_2}}{d} \label{eq:submult}\\
    &<\varepsilon, \label{eq:assumption}
    \end{align}
    \endgroup
    where \eqref{eq:triangle-1} and \eqref{eq:triangle-2} use the triangle inequality, \eqref{eq:unitary-invar} uses the fact that Schatten norms are invariant under unitary multiplications, \eqref{eq:submult} uses the submultiplicativity of the Schatten norm, and \eqref{eq:assumption} uses our assumption that \eqref{eq:close-to-id} holds for all $i \in [d]$. By taking the maximum over all such $X$, we conclude that for sufficiently small $s$, the tuple $\vU^s$ is an $(n,d,\mu)$-dimension expander, but $\lambda(\Phi_{\vU^s}) < \varepsilon$.
\end{proof}

In order to prove \cref{thm:exp-keeps-dim-exp}, we will need the following lemma. It says that for a \emph{fixed} subspace $W \leq \bC^n$, taking a small power of a matrix $U$ does not ``ruin the dimension expansion''. Namely, if there is a subspace $V \leq W$ so that $UV \cap W = \{0\}$, then $U^s V \cap W=\{0\}$ for all sufficiently small $s$. This almost suffices to prove \cref{thm:exp-keeps-dim-exp}, except that the ``sufficiently small'' condition depends on $V$ and $W$; in order to obtain an absolute bound that is independent of these choices, we will use a compactness argument.
\begin{lemma}\label{lem:small-power}
    Let $U$ be a unitary matrix and let $V \leq W \leq \bC^n$ be subspaces. Suppose that $Uv \notin W$ for all non-zero $v \in V$. There exists some $\delta>0$ so that $U^s v \notin W$ for all $0<s<\delta$.
\end{lemma}
Here, as before, the matrix power $U^s$ is defined by fixing a Hermitian matrix $H$ so that $U=e^{iH}$, and then defining $U^s \coloneqq e^{isH}$.
At a high level, \cref{lem:small-power} is proved as follows. For a fixed non-zero $v \in V$, we know that $Uv \notin W$. Consider the ``trajectory'' $U^s v$, viewed as a function of $s$. We know that this trajectory eventually leaves $W$ (as $Uv \notin W$); let us assume for the moment that in fact, the first derivative of the trajectory has a non-zero component in $W^\perp$. This means that for sufficiently small $s$, the trajectory $U^s v$ is well-approximated by a line which does not lie in $W$, and thus for some small but positive amount of time, this trajectory stays outside $W$. By bounding the second derivative of the trajectory, this argument can be made rigorous. 

Unfortunately, it need not be the case that the first derivative of the trajectory lies outside $W$, so one cannot just apply the argument as sketched above. However, since $Uv \notin W$ for all $v \in V$, we see that for each $v \in V$, \emph{some} derivative of the trajectory lies outside $W$. By performing a backwards induction on the order of the first derivative which lies outside $W$ (and at each step repeating the argument above), one can prove \cref{lem:small-power}. Here are the details.
\begin{proof}[Proof of \cref{lem:small-power}]
    Fix a non-zero vector $v \in V$. For $s \geq 0$, we write $U^s v$ in terms of the Taylor expansion of $e^{isH}$, i.e.
    \[
        U^s v = e^{isH} v = \sum_{j=0}^\infty \frac{1}{j!} (isH)^j v= \sum_{j=0}^\infty \frac{(is)^j}{j!} H^j v.
    \]
    Note that if $H^k v\in W$ for all $k \geq 0$, then $U^s v \in W$ for all $s \geq 0$, and in particular $Uv \in W$, which contradicts our assumption. Therefore, we see that for every $v \in V$, there is some integer $k \geq 0$ so that $H^j v \in W$ for all $0 \leq j \leq k$, but $H^{k+1}v \notin W$.

    For $k \geq 0$, let $V_k \leq V$ be the set of $v \in V$ so that $H^j v \in W$ for all $0 \leq j \leq k$. Note that $V_k$ is a subspace of $V$, and they are nested as $V=V_0 \geq V_1 \geq V_2 \geq \dotsb$. By the discussion above, we know that $\bigcap_{k \geq 0} V_k = \{0\}$. Additionally, since $V$ is finite-dimensional, we see that this chain eventually stabilizes at $0$, i.e., there is some $K$ so that $V_{K-1} \neq \{0\}$ but $V_{K}=\{0\}$. We will now prove the following claim by induction on $k$, starting at $k=K$ and working down to $k=1$.
    \begin{claim}\label{claim:induction}
        For every $1 \leq k \leq K$, there exists some $\delta_k>0$ so that the following holds. For every $v \in V_{k-1}$ with $\norm{v}_{L_2}=1$, we have that $U^s v \notin W$ for all $s \in (0,\delta_k)$.
    \end{claim}
    Note that the $k=1$ case of this claim is simply the desired lemma statement, since $V_0=V$ and since we lose nothing by restricting to vectors of norm $1$.
    \begin{proof}[Proof of \cref{claim:induction}]
        We prove the claim by (reverse) induction on $k$. For the base case of $k=K$, let $v \in V_{K-1}$ be a vector of norm $1$. Let $P\in\M(n,\bC)$ denote the orthogonal projection onto the orthogonal complement of $W$. Observe that for any $s>0$, we have
        \[
            P U^s v = \sum_{j=0}^\infty \frac{(is)^j}{j!} PH^j v = \sum_{j=K}^\infty \frac{(is)^j}{j!} PH^j v,
        \]
        since $H^jv \in W$ for all $j<K$, and thus $PH^j v =0$ for all $j <K$. For any $s \in (0,1)$, we have that
        \[
            \bnorm{\frac{(is)^{K+1}}{(K+1)!} PH^{K+1}v}_{L_2} = \frac{s^{K+1}}{(K+1)!} \norm{PH^{K+1}v}_{L_2} \leq \frac{1}{(K+1)!} \norm{H^{K+1}v}_{L_2} \leq \frac{\norm H_{S_{\infty}}^{K+1}}{(K+1)!}\eqqcolon C_K,
        \]
        where we use that $s \leq 1$ and that $P$ is a contraction in the first inequality, and the definition of the Schatten-$\infty$ norm and the assumption $\norm{v}_{L_2}=1$ in the second inequality. Therefore, by Taylor's theorem, we find that for all $s \in (0,1)$ and all $v \in V_{K-1}$ with $\norm{v}_{L_2}=1$, we have that
        \[
            \norm{PU^s v}_{L_2} \geq \bnorm{\frac{(is)^K}{K!} PH^K v}_{L_2} - C_K s^{K+1} = \frac{s^K}{K!} \norm{PH^K v}_{L_2} - C_K s^{K+1}.
        \]
        Note that the function $v \mapsto \norm{PH^K v}_{L_2}$ is a continuous real-valued function on the unit sphere in $V_{K-1}$, which is compact. Moreover, since $V_K=\{0\}$, we know that $H^K v \notin W$ for all non-zero $v \in V_{K-1}$, and thus $\norm{PH^K v}_{L_2}$ is strictly positive for all $v \in V_{K-1}$ with $\norm v_{L_2}=1$. Therefore, there exists some $c_K>0$ so that $\norm{PH^K v}_{L_2} \geq K! c_K$ for all $v \in V_{K-1}$ with $\norm v_{L_2}=1$. Continuing our computation above, we conclude that

         \[
            \norm{PU^s v}_{L_2} \geq c_K s^K - C_K s^{K+1}
        \]
        for all $s \in (0,1)$ and all $v \in V_{K-1}$ with $\norm v_{L_2}=1$. If we let $\delta_K=\min\{c_K/C_K,1\}$, then this implies that $\norm{PU^sv}_{L_2}>0$ for all $s \in (0,\delta_K)$. This is equivalent to saying that $U^s v \notin W$ for all $s \in (0,\delta_K)$, which proves the claim for $k=K$.

        We now move to the inductive step. There is nothing to prove if $V_{k-1}=V_k$, so we may assume that $V_k$ is a proper subspace of $V_{k-1}$. Inductively, suppose we know the claim holds for $k+1$, namely that there exists some $\delta_{k+1}>0$ so that $U^s v \notin W$ for all $v \in V_k$ with $\norm{v}_{L_2}=1$ and all $s \in (0,\delta_{k+1})$. As $v \mapsto U^s v$ is a continuous map, and as $W$ is closed, we conclude that $U^s v \notin W$ for all $s \in (0,\delta_{k+1})$ and all $v$ which is in a sufficiently small open neighborhood of the unit sphere in $V_k$. More precisely, there exists some $\varepsilon_k>0$ so that the following holds for all $v \in V_{k-1}$ with $\norm v_{L_2}=1$: Suppose we write $v=u+w$ where $u \in V_{k}^\perp$ and $w \in V_{k}$, and suppose that $\norm{u}_{L_2} <\varepsilon_k$. Then $U^s v \notin W$ for all $s \in (0,\delta_{k+1})$.

        So it suffices to now only consider such $v$ with $\norm u_{L_2} \geq \varepsilon_k$. Note that for any such $v$, we have that 
        \[
            \norm{PH^k v}_{L_2} = \norm{PH^k (u+w)}_{L_2} = \norm{PH^k u}_{L_2},
        \]
        since $H^k w \in W$, as $w \in V_k$. Now, the set of unit vectors $v \in V_{k-1}$ for which $\norm u_{L_2} \geq \varepsilon_k$ is compact, and the function $v \mapsto \norm{PH^k v}_{L_2}$ is continuous and strictly positive on it. So there exists some $c_k>0$ so that $\norm{PH^k v}_{L_2} \geq k! c_k$ for all such $v$.
        The rest of the proof is very similar to the base case. For any $s>0$, we have that
        \[
            PU^sv = \sum_{j=0}^\infty \frac{(is)^j}{j!} PH^j v = \sum_{j=k}^\infty \frac{(is)^j}{j!} PH^j v.
        \]
        For any $s \in (0,1)$, we have that 
        \[
            \bnorm{\frac{(is)^{k+1}}{(k+1)!} PH^{k+1}v}_{L_2} = \frac{s^{k+1}}{(k+1)!} \norm{PH^{k+1}v}_{L_2} \leq \frac{1}{(k+1)!} \norm{H^{k+1}v}_{L_2} \leq \frac{\norm H_{S_\infty}^{k+1}}{(k+1)!}\eqqcolon C_k.
        \]
        By Taylor's theorem, we conclude that if $v=u+w$ is such that $\norm u_{L_2} \geq \varepsilon_k$, then for any $s \in (0,1)$,
        \[
            \norm{PU^s v}_{L_2} \geq \bnorm{\frac{(is)^k}{k!} PH^k v}_{L_2} - C_k s^{k+1} = \frac{s^k}{k!} \norm{PH^k v}_{L_2} - C_k s^{k+1} \geq c_k s^k - C_k s^{k+1}.
        \]
        Thus, for such $v$, we see that $U^s v \notin W$ for all $s \in (0, c_k/C_k)$. On the other hand, for those $v$ with $\norm u<\varepsilon_k$, we know that $U^s v \notin W$ for all $s \in (0,\delta_{k+1})$. Thus, we get the desired result by setting $\delta_k = \min\{c_k/C_k, \delta_{k+1},1\}$.
    \end{proof}
    As discussed above, the $k=1$ case of the claim is equivalent to the lemma statement, so this concludes the proof. 
\end{proof}

For $1 \leq r \leq n$, let $\Gr(n,r)$ denote the Grassmannian of $r$-dimensional subspaces of $\bC^n$. We make the following definition, which will be useful in the proof of \cref{thm:exp-keeps-dim-exp}.
\begin{definition}
    Let $\vU=(U_1,\dots,U_d)\in\mathrm{U}(n)^d$ be a tuple of $n \times n$ unitary matrices. 
    Given a real number $\mu>0$ and a subspace $W \in \Gr(n,r)$, let us say that a tuple $(V,i,\delta)$ is $\mu$-\emph{expansive for $W$} if it satisfies the following conditions.
    \begin{enumerate}[label=(\roman*)]
        \item $V$ is a subspace of $W$, $i \in [d]$ is an integer, and $\delta>0$ is a strictly positive real number.
        \item We have $\dim V \geq \mu r/d$.
        \item For every $s \in (0,\delta)$, we have that $U_i^s V \cap W=\{0\}$. 
    \end{enumerate}
\end{definition}
Our next simple lemma shows that if $\vU$ is a $(n,d,\mu)$-dimension expander, then every subspace has an expansive tuple. The implication is a straightforward consequence of \cref{lem:small-power}, but the language of expansive tuples will be more convenient for the compactness argument we use in the proof of \cref{thm:exp-keeps-dim-exp}.
\begin{lemma}\label{lem:expansive-exists}
    Let $\vU=(U_1,\dots,U_d)\in\mathrm{U}(n)^d$ be a unitary matrix tuple, and let $\mu>0$ be a real number. If $\vU$ is a $(n,d,\mu)$-dimension expander, then for all $1 \leq r \leq n/2$ and all $W \in \Gr(n,r)$, there is a $\mu$-expansive tuple for $W$.
\end{lemma}
\begin{proof}
    By the definition of dimension expansion, we know that $\dim(W+U_1 W + \dotsb + U_d W) - \dim(W) \geq \mu r$. Therefore, there exists some $i$ so that $\dim(W+U_i W)-\dim(W) \geq \mu r/d$. Let $V$ be a maximum-dimensional subspace of $W$ with the property that $U_i V \cap W = \{0\}$; then the above implies that $\dim(V) \geq \mu r/d$. Finally, by \cref{lem:small-power}, we see that there exists some $\delta>0$ so that $U^s V \cap W = \{0\}$ for all $0<s<\delta$, implying that $(V,i,\delta)$ is $\mu$-expansive for $W$.
\end{proof}

Now suppose we are given a unitary matrix $U\in \U(n)$ and a subspace $V \leq W$ with $U^s V \cap W = \{0\}$ for all $0<s<\delta$. Intuitively, the continuity of the map $V \mapsto U^s V$ implies that if we perturb $W$ to a ``nearby'' subspace $W'$, we can similarly perturb $V$ to $V' \leq W'$ with the property that $U^s V' \cap W' =\{0\}$ for all $0<s<\delta$. The following lemma makes this precise, for which it is best to use the language of fiber bundles.
Let $\Gr(n,{\leq r})$ denote the disjoint union of $\Gr(n,\ell)$ over $0 \leq \ell \leq r$. There is a fiber bundle $\pi$ over $\Gr(n,r)$ whose fibers are $\Gr(r,{\leq r})$, namely above $W \in \Gr(n,r)$ we simply put all possible subspaces of $W$. More precisely, the total space of the bundle is
\[
    E = \left\{(W,V) \in \Gr(n,r) \times \Gr(n,{\leq r})  : V \text{ is a subspace of } W\right\},
\]
and the bundle map $\pi:E \to \Gr(n,r)$ is given by $\pi(W,V)=W$.

\begin{lemma}\label{lem:section-exists}
    Let $W \in \Gr(n,r)$, and let $U$ be an $n\times n$ unitary matrix. Suppose that there exist $\delta>0$ and a subspace $V \leq W$ so that $U^s V \cap W = \{0\}$ for all $0<s<\delta$. Then there exists an open set $O \subseteq \Gr(n,r)$ with $W \in O$ and a continuous section $\sigma:O \to E$ of the fiber bundle $\pi$ so that $\sigma(W)=V$ and for all $W' \in O$, we have that $U^s \sigma(W') \cap W' = \{0\}$. 
\end{lemma}
\begin{proof}
    Let $\ell = \dim V$. For $X\in \Gr(n,r)$, let $\V(X)$ be the collection of $\ell$-dimensional subspaces $Y$ of $X$ with the property that $U^s Y \cap X = \{0\}$ for all $0<s<\delta$. As $(s,Y) \mapsto U^s Y$ is continuous in both variables, we see that $\V(X)$ is open for all $X$. Again by continuity, $\V(X)$ also varies continuously as we vary $X \in \Gr(n,r)$. Since $V \in \V(W)$, these properties imply that we can find an open neighborhood $O$ of $W$ and a section $\sigma$ as claimed. 
\end{proof}


We are now ready to prove \cref{thm:exp-keeps-dim-exp}.
\begin{proof}[Proof of \cref{thm:exp-keeps-dim-exp}]
    Let $\vU=(U_1,\dots,U_d)\in\mathrm{U}(n)^d$ be a $(n,d,\mu)$-dimension expander. By \cref{lem:expansive-exists}, for every $1 \leq r \leq n/2$ and every $W \in \Gr(n,r)$, we may find an expansive tuple $(V,i,\delta)$ for $W$. By \cref{lem:section-exists}, there exists an open neighborhood $O_W$ of $W$ as well as a section $\sigma:O_W \to E$ of the bundle $\pi$ so that $U_i^s \sigma_i(W') \cap W'=\{0\}$ for all $W' \in O_W$ and all $0<s<\delta$.

    Now, the collection $\{O_W\}_{W \in \Gr(n,r)}$ forms an open cover of $\Gr(n,r)$, so by compactness, we can find a finite subcover, say $O_1,\dots,O_T$. By the way we constructed these $O_W$, we see that there are $\delta_1,\dots,\delta_T>0$ so that for each $W \in N_j$, there is an expansive tuple for $W$ with $\delta=\delta_j$. By letting $s_0 = \min_{j} \delta_j$ we conclude that for every $W \in \Gr(n,r)$, there is an expansive tuple for $W$ with $\delta \geq s_0$. In other words, for every $W$, there exist $i \in [d]$ and $V\leq W$ with $\dim V \geq \mu r/d$ so that $U_i^s V \cap W=\{0\}$ for all $0 <s <s_0$. This implies that $\dim(W+\vU^s(W))-\dim(W)\geq \mu r/d$ for all $W$. In other words, we see that $\mu(\vU^s)\geq \mu/d$ for all $0<s<s_0$, as claimed.
\end{proof}


To summarize, we have proven that given \emph{any} dimension expander $\vU=(U_1,\dots,U_d) \in \U(n)^d$, two things are simultaneously true. On the one hand, all sufficiently small powers $\vU^s$ remain dimension expanders. On the other hand, as $s \to 0$, the tuple $\vU^s$ converges to the identity tuple, and thus a sufficiently small power is an arbitrarily bad quantum expander. It is natural to hope that one can ``reverse'' this process; namely, that by taking a \emph{large} power $s$, we can convert any dimension expander into one that is also a quantum expander. Sadly, this is also not true, as shown by the following simple counterexample. 
\begin{proposition}
	There exists some dimension expander $\vU \in \U(n)^d$ such that for any $s>0$, ${\vU^s}$ is not a quantum expander. 

    More precisely, there exists an absolute constant $\mu >0$ such that the following holds for all $\varepsilon>0$ and all sufficiently large $n$. There exists a $(n,100,\mu)$-dimension expander $\vU = (U_1,\dots,U_{100}) \in \U(n)^{100}$ such that for all $s>0$, we have $\lambda({\vU^s})< \varepsilon$.
\end{proposition}
\begin{proof}[Proof sketch]
Fix some $\varepsilon>0$. Let $e_1,\dots,e_n$ be the standard basis of $\bC^n$. 
Let $E_\varepsilon \subseteq \mathrm U(n)$ denote the set of $n\times n$ unitary matrices $M$ with the property that for all $j\in[n]$, we have $\ab{\langle m_j,e_j\rangle}> 1- \varepsilon$, where $m_j$ is the $j$th column of $M$. Then $E_\varepsilon$ is a non-empty open subset of $\mathrm U(n)$, which means that we can sample according to the induced Haar measure on $E_\varepsilon$. Let $M_1,\dots,M_{100}$ be $100$ independently random samples from this measure. Additionally, let $D_1,\dots,D_{100}$ be independent random diagonal matrices whose diagonal entries are drawn uniformly at random from the unit circle. Finally, let $U_i = M_i D_i M_i^*$, so that each $U_i$ is a random unitary matrix whose eigenvectors are the columns of $M_i$ and whose eigenvalues are the diagonal entries of $D_i$. Note that for any $s>0$, we have $U_i^s = M_i D_i^s M_i^*$, and $D_i^s$ is a diagonal matrix whose diagonal entries are the $s$th powers of the diagonal entries of $D_i$.

It is well-known that $100$ random unitary matrices form a $(n,100,\mu)$-dimension expander for some fixed $\mu>0$ when $n$ is large~\cite{DS09}. For the same reason, it is straightforward to check that $\vU=(U_1,\dots,U_{100})$ forms an $(n,100,\mu)$-dimension expander for some fixed $\mu>0$. The point is that while $U_1,\dots,U_d$ are not uniformly random unitary matrices, they are ``generic'' in an appropriate sense, which suffices for them to form a dimension expander. However, we claim that for any $s>0$, we have $\lambda(\vU^s)\leq 10 \varepsilon$. Since $\varepsilon$ was arbitrary, this yields an example of a dimension expander none of whose powers is a quantum expander. 

To see this, we first observe that by the definition of $E_\varepsilon$, the $(1,1)$ entry of $M_i$ has absolute value at least $1- \varepsilon$, and the first row of $M_i$ is a unit vector. Since $D_i^s$ is a diagonal matrix whose diagonal entries have absolute value $1$, this implies that both $M_i^* e_1$ and $D_i^s M_i^* e_1$ are unit vectors whose first coordinate has norm at least $1- \varepsilon$. Let $v,w \in \bC^{n-1}$ be the last $n-1$ coordinates of $M_i^* e_1$ and $D_i^s M_i^* e_1$, respectively, so that $\norm v_{L_2},\norm w_{L_2} \leq \sqrt{1-(1- \varepsilon)^2} \leq \sqrt{2 \varepsilon}$.
The Cauchy--Schwarz inequality then gives $\ab{\langle v,w\rangle} \leq 2 \varepsilon$, which implies implies
\[
    \ab{\langle e_1,U_i^s e_1\rangle} = \ab{\langle e_1, M_i D_i^s M_i^* e_1\rangle} = \ab{\langle M_i^* e_1, D_i^s M_i^* e_1\rangle} \geq (1- \varepsilon)^2 - \ab{\langle v,w\rangle}  \geq 1- 4 \varepsilon.
\]
Let $P = e_1e_1^*$ be the projection on to the subspace spanned by $e_1$. Then the entry of $U_i^s P (U_i^s)^*$ in the $(1,1)$ position is
\[
    e_1^* \left(U_i^s P (U_i^s)^*\right) e_1 = (e_1^* U_i^s e_1) (e_1^* (U_i^s)^* e_1) = \ab{\langle e_1,U_i^s e_1\rangle}^2 \geq 1- 8 \varepsilon.
\]
Let $X = P - \frac1 n I_n$, so that $X$ is a traceless matrix with $S_2$-norm $1-O(\frac 1n)$. The computation above implies that the $(1,1)$ entry of $U_i^s X (U_i^s)^*$ is at least $1-8 \varepsilon- \frac 1n$. As this holds for all $i$, we conclude that it also holds for $\Phi_{\vU^s}(X)$, which in turn implies that $\norm{\Phi_{\vU^s}(X)}_{S_2} \geq 1- 8 \varepsilon - \frac 1n$. As $\norm X_{S_2} \geq 1- O(\frac 1n)$, we conclude that $\lambda(\vU^s) \leq 10 \varepsilon$ for all sufficiently large $n$.
\end{proof}


\section{Relations between linear-algebraic notions of expansion}\label{sec:quantum-to-dim}
In this section, we prove \cref{thm:main-relationships}. \cref{thmit:quant-cheeger} was proved by Hastings \cite[Appendix A]{PhysRevA.76.032315}, so it remains to prove \cref{thmit:dim,thmit:edge-comparison}.

We begin with \cref{thmit:dim}. 
As remarked in the introduction, the result actually holds for arbitrary matrix tuples over arbitrary fields, as stated in the following result.


\begin{proposition}\label{prop:dim_to_rank2}
	For $\vB\coloneqq (B_1, \dots, B_d)\in \M(n, \F)^d$, it holds that $\frac {\mu(\vB)}d\leq h_D(\vB)\leq \mu(\vB)$.
\end{proposition}
\begin{proof}
	We first show that an $\mu(\vB)\leq d\cdot h_D(\vB)$. For any $\vB\in \M(n,\F)^d$, let $\vB'=(\vB,I_n)\in \M(n,\F)^{d+1}$. From the definitions, it is clear that $\mu(\vB')=\mu(\vB)$ and $h_D(\vB')=\frac{d}{d+1} h_D(\vB)$. We shall prove that $\mu(\vB')\leq (d+1)h_D(\vB')$.
	
    Fix some $V\leq \F^n$ of dimension $1 \leq r\leq n/2$ and let $T\in \M(n\times r,\F)$ be a matrix whose columns form a basis of $V$. Let $R\in \M((n-r)\times n,\F)$ be a matrix whose rows form a basis of $\ann$, as defined in~\cref{rmk:def-annihilator}. Let $B_i|_{\ann,V}=RB_iT\in\M((n-r)\times r,\F)$ for each $i\in[d]$. Note that we have $\rk(T)=r$ and $\rk(R)=n-r$.
	
	Let $t=\dim(V+\vB'(V))$, and notice that $t=\dim(\vB'(V))$ since $B_{d+1}=I_n$. We have that $t\geq (1+\mu(\vB'))\cdot\dim(V)=(1+\mu(\vB'))\cdot r$. Let $W\in \M(n\times t,\F)$ be a matrix whose columns form a basis of $\vB'(V)$. 
    Then we have $\rk(RW)=\dim(\vB'(V))-\dim(\ker(R)\cap \vB'(V))\geq t-r \geq \mu(\vB')\cdot r$ and
	\begin{align}
		\sum_{i=1}^{d+1}\rk(B_i|_{\ann, V}) & =  \sum_{i=1}^{d+1}\dim(\colspan(B_i|_{\ann, V})) \nonumber \\
		& \geq \dim(\langle\cup_{i\in[d+1]} \colspan(RB_i T)\rangle) \nonumber \\ 
		& = \dim(R\langle\cup_{i\in[d+1]} \colspan(B_iT)\rangle)\label{eq:rank-R} \\
		& = \rk(RW) \label{eq:colspan} \\
		& \geq \mu(\vB^{\prime})\cdot r,\nonumber
	\end{align}
    where \eqref{eq:rank-R} holds since $R$ has full row rank and \eqref{eq:colspan} holds  since $\langle \cup_{i\in[d+1]} \colspan(B_iT)\rangle=\vB'(V)$.
	Ranging over all subspaces $V\leq\F^n$ of dimension at most $n/2$, 
	we find that $\mu(\vB')\leq (d+1) h_D(\vB')$, and $\mu(\vB)\leq d\cdot h_D(\vB)$ follows.
	
	Now we show that $ h_D(\vB)\leq \mu(\vB)$. Fix some $V\leq \F^n$ of dimension $1\leq r\leq n/2$, let $T,R$ be as above, and note that $V=\ker(R)$. Let $W$ be a matrix whose columns are a basis of $\vB(V)$. 
	We have that 
	\begin{align*}
		d\cdot h_D(\vB&)\cdot\dim(V)  \leq  \sum_{i=1}^d\rk(B_i|_{\ann, V}) \\
		& =  \sum_{i=1}^d\dim(\colspan(RB_iT)) \\
		& \leq d\cdot\dim(\cup_{i\in[d]}\colspan(RB_iT)) & [\text{since }\textstyle\sum_{i\in[d]}\dim(W_i)\leq d\cdot \dim(\cup_{i\in[d]}W_i) ]\\
		& =  d\cdot\dim(R(\cup_{i\in[d]}\colspan(B_iT)))& [\text{since }R\text{ has full row rank} ]\\
		& =  d\cdot\rank(RW) &[\text{since }\cup_{i\in[d]} \colspan(B_iT)=\vB(V) ]\\
		& =  d\cdot(\dim(\vB(V))-\dim(\vB(V) \cap \ker(R)) & [\text{since }\vB(V)=\colspan(W)]\\
		& =  d\cdot(\dim(\vB(V))-\dim(V\cap\vB(V))) & [\text{since }V=\ker(R)]\\
		& =  d\cdot(\dim(V+\vB(V))-\dim(V)).
	\end{align*}
	This implies that $\dim(V+\vB(V))-\dim (V)\geq h_D(\vB)\cdot \dim(V)$. Ranging over all subspace $V\leq\F^n$ of dimension at most $n/2$,
	we conclude that $ h_D(\vB)\leq \mu(\vB)$.
\end{proof} 

We now turn to \cref{thmit:edge-comparison,thmit:unitary-edge-comparison}, which lower-bound the dimension edge expansion in terms of the quantum edge expansion. As explained in the introduction, it is this inequality which allows us to prove that quantum expanders are dimension expanders.







In order to prove \cref{thmit:edge-comparison,thmit:unitary-edge-comparison}, we will use the following equivalent formulation of quantum edge expansion.
\begin{lemma}
	For any doubly stochastic matrix tuple $\vB=(B_1,\dots,B_d) \in \M(n,\bC)^d$, we have
	\begin{equation}\label{eq: quantum edge expansion norm formulation}
		h_Q(\vB)=\min_{\substack{V \leq \bC^n\\1\leq\dim(V)\leq\frac{n}{2}}}\frac{\sum_{i=1}^d\norm{B_i|_{V^\perp,V}}_{S_2}^2}{d\cdot\dim(V)}.
	\end{equation}
\end{lemma}
\begin{proof}
    Fix some subspace $V \leq \bC^n$ of dimension $1 \leq r \leq n/2$.
    Let $T_V$ be an $n\times r$ matrix whose columns form an 
	orthonormal basis of $V$. Then $P_V=T_V T^*_V$. 
    Let 
	$V^\perp$ be the orthogonal complement of $V$, and $T_{V^\perp}$ be an 
	$n\times (n-r)$ matrix whose columns form an orthonormal basis of $V^\perp$. Then 
	$P_{V^\perp}=I_n-P_V=T_{V^\perp}T_{V^\perp}^*$.
	
	We then have 
	\[
	\begin{split}
		\langle I_n-P_V,\Phi(P_V)\rangle&=\trace((I_n-P_V)^*\Phi(P_V))\\
		&=\frac 1d\sum_{i=1}^d\trace\left(T_{V^\perp}
                T_{V^\perp}^* B_iT_VT^*_VB_i^*\right)\\
		&=\frac 1d \sum_{i=1}^d\trace\left(T_{V^\perp}^*B_iT_VT^*_VB_i^*T_{V^\perp}\right)\\
		&=\frac 1d\sum_{i=1}^d\norm{T_{V^\perp}^*B_iT_V}_{S_2}^2\\
        &=\frac 1d\sum_{i=1}^d\norm{B_i|_{V^\perp, V}}_{S_2}^2.
	\end{split}
	\]
    This implies that the objective functions in~\eqref{eq: quantum edge expansion norm formulation} and \eqref{eq: quantum edge expansion} are identical. The feasible regions are also the same, which concludes the proof.
\end{proof}

We are now ready to prove \cref{thmit:edge-comparison,thmit:unitary-edge-comparison}.
\begin{proof}[Proof of~\cref{thmit:edge-comparison,thmit:unitary-edge-comparison}]
    Fix a doubly stochastic matrix tuple $\vB = (B_1,\dots,B_d) \in \M(n,\bC)^d$. By \eqref{eq: quantum edge expansion norm formulation} and \eqref{eq: dim edge expansion}, we wish to prove that
    \[
        \min_{\substack{V \leq \bC^n\\1\leq\dim(V)\leq\frac{n}{2}}}\frac{\sum_{i=1}^d\norm{B_i|_{V^\perp,V}}_{S_2}^2}{d\cdot\dim(V)} \leq
        d \cdot 
        \min_{\substack{V \leq \bC^n\\
        1\leq\dim(V)\leq \frac n2}}\frac{\sum_{i=1}^d\rk(B_i|_{V^\perp, V})}{d\cdot\dim(V)}.
    \]
    So it suffices to prove that
    $\norm{B_i|_{V^\perp, V}}_{S_2}^2\leq d\cdot \rk(B_i|_{V^\perp, V})$ for any subspace $V\leq\bC^n$ and all $i\in[d]$. We first claim that the operator norm of $B_i|_{V^\perp, V}=T_{V^\perp}^*B_iT_V$ is at most $\sqrt d$. Indeed, recall that $\sum_{i=1}^d B_i^*B_i=dI_n$, so $dI_n-B_i^*B_i$ is positive semidefinite for any $i\in[d]$. Thus, the operator norm of $B_i$ is upper bounded by $\sqrt d$ for each $i\in[d]$. Moreover, the operator norm of any isometry is at most $1$, thus $\norm{T_{V^\perp}}_{S_\infty},\norm{T_V}_{S_\infty}\leq 1$. Using the submultiplicativity of the operator norm, we have  $\norm{T_{V^\perp}^*B_iT_V}_{S_\infty}\leq \sqrt d$. 

    Recall that $ \norm{B_i|_{V^{\perp},V}}_{S_2}^2 = \norm{T_{V^\perp}^*B_i T_V}_{S_2}^2$ is the sum of the squares of the singular values of $T_{V^\perp}^*B_i T_V$. As there are precisely $\rk(T_{V^\perp}^*B_i T_V)$ non-zero singular values, and each one is upper-bounded by $\norm{T_{V^\perp}^*B_iT_V}_{S_\infty}\leq \sqrt d$, we conclude that 
	\[
	\norm{B_i|_{V^\perp, V}}_{S_2}^2\leq \rk(T_{V^\perp}^*B_iT_V)\norm{T_{V^\perp}^*B_iT_V}_{S_\infty}^2\leq d\cdot \rk(B_i|_{V^\perp, V}),
	\]
	as claimed. This proves \cref{thmit:edge-comparison}.

    In order to prove \cref{thmit:unitary-edge-comparison}, note that 
    if $B_i$ is a unitary matrix, then $\norm{B_i}_{S_\infty}=1$. This implies that $\norm{T_{V^\perp}^* B_i T_V}_{S_\infty}^2 \leq 1$, so the argument above shows that $\norm{B_i|_{V^\perp, V}}_{S_2}^2\leq \rk(B_i|_{V^\perp, V})$, which yields \cref{thmit:unitary-edge-comparison}.
\end{proof}


\section{Connections between graphs and matrix spaces}\label{sec:connections}
In this section, we study the graphical matrix tuple $\vB_G$ associated to a $d$-regular graph $G$.
We begin by proving that $h(\vB_G)$ is in general different from $h(G)$, as stated in \cref{prop:hq-k2}.
\begin{proof}[Proof of \cref{prop:hq-k2}]
    We certainly have that $h(K_2)=1$. Note that $\vB_{K_2} = \left(\smat{0&\sqrt 2\\0&0}, \smat{0&0\\\sqrt 2&0}\right)$.

    Consider the subspace $V \leq \bC^2$ spanned by the vector $(\frac{1}{\sqrt 2},\frac{1}{\sqrt 2})$. The orthogonal projection onto $V$ is given by the matrix $P_V = \frac 12\smat{1&1\\ 1 & 1}$, and
    \[
        \Phi_{\vB_{K_2}} (P_V) = \frac 12 \left(
        \begin{bmatrix}
            0&\sqrt 2\\0&0
        \end{bmatrix}
        \begin{bmatrix}
            \frac 12 & \frac 12 \\ \frac 12&\frac 12
        \end{bmatrix}
        \begin{bmatrix}
            0&\sqrt 2\\0&0
        \end{bmatrix}^*
        +
        \begin{bmatrix}
            0&0\\\sqrt 2&0
        \end{bmatrix}
        \begin{bmatrix}
            \frac 12 & \frac 12 \\ \frac 12&\frac 12
        \end{bmatrix}
        \begin{bmatrix}
            0&0\\\sqrt 2&0
        \end{bmatrix}^*
        \right)
        =
        \begin{bmatrix}
            \frac 12 &0\\0&\frac 12
        \end{bmatrix}.
    \]
    Therefore,
    \[
        \langle I_2-P_V, \Phi_{\vB_{K_2}}(P_V)\rangle = \trace \left(
        \begin{bmatrix}
            \frac 12 & -\frac 12 \\ -\frac 12&\frac 12
        \end{bmatrix}^*
        \begin{bmatrix}
            \frac 12 &0\\0&\frac 12
        \end{bmatrix}
        \right)
        =\frac 12.
    \]
    Hence, as $h_Q(\vB_{K_2})$ is defined as a minimum over all one-dimensional subspaces, we find that $h_Q(\vB_{K_2}) \leq \langle I_2-P_V, \Phi_{\vB_{K_2}}(P_V)\rangle = \frac 12$, as claimed.
\end{proof}
\begin{remark}
It is not hard to show that in fact, $h_Q(\vB_{K_2})=\frac 12$. More generally, one can show that $h_Q(\vB_{K_n})\leq\frac 1n$, which is smaller than $h(K_n)$ for all $n\geq2$. 
\end{remark}
We now turn to the proof of \cref{thm:graphical-tuple-preserves}. As the proofs that $h_D(\vB_G)=h(G)$ and $\mu(\vB_G)=\mu(G)$ are disjoint, we separate the \cref{thm:graphical-tuple-preserves} into two statements, \cref{prop:graphical-edge,prop:graphical-vertex}.
Note that \cref{thm:graphical-tuple-preserves} holds over any field, so we will work with $\mathbb{F}$ instead of $\mathbb{C}$ in the rest of this section. 
\begin{proposition}\label{prop:graphical-edge}
    For any $d$-regular graph $G=([n],E)$, we have $h_D(\vB_G)=h(G)$.
\end{proposition}
\begin{proof}
    For $v\in\F^n$, denote by $\supp(v)\subseteq[n]$ the set of indices of the non-zero coordinates of $v$. For $V\leq\F^{n}$, we let $\supp(V)=\cup_{v\in V}\supp(v)$. 
    Define $\ann$ as in~\cref{rmk:def-annihilator}. Now we claim the following.
    \begin{claim}\label{claim:rank-Eij}
        We have that
        \[
            \rk(\E_{i,j}|_{\ann,V}) = 
            \begin{cases}
                1&\text{if }i \in \supp(\ann)\text{ and } j \in \supp(V)\\
                0&\text{otherwise.}
            \end{cases}
        \]
    \end{claim}
    \begin{proof}
        Suppose $\dim(V)=r$ and $\dim(\ann)=n-r$. Let $T_V$ (resp.\ $T_{\ann}$) be an $n\times r$ (resp.\ $n\times (n-r)$) matrix whose columns form a basis of $V$ (resp.\ $\ann$). Denote by $v_{1},\dots,v_{n}\in \F^r$ and $v^{\prime}_{1},\dots,v^{\prime}_{n}\in\F^{n-r}$ the vectors corresponding to the rows of $T_V$ and $T_{\ann}$, respectively. Then for any $i,j\in[n]$,
        $$
        \E_{i,j}|_{\ann,V}=T_{\ann}^t\E_{i,j}T_V={v^{\prime}_{i}}v_{j}^{t}.
        $$ 
        Note that $\rk(\E_{i,j}|_{\ann,V})=1$ if and only if $\E_{i,j}|_{\ann,V}\neq0$ (and otherwise $\rk(\E_{i,j}|_{\ann,V})=0$). This, in turn, happens if and only if $v^{\prime}_{i}\neq0$ and $v_{j}\neq0$, which is equivalent to $i\in\supp(\ann)$ and $j\in\supp(V)$. 
    \end{proof}

    For any fixed $V\leq\F^n$ of dimension $1\leq r\leq n/2$, we shall construct a vertex subset $W\subseteq[n]$ of size $r$ such that 
    \[
        \sum_{\{i,j\}\in E}\rk(\E_{i,j}|_{\ann, V})\geq |\partial W|,
    \]
    where here and throughout the sum is over all \emph{ordered} pairs of vertices which are adjacent in $G$.
    We use the same notation of $T_V$ and $T_{\ann}$ as in \cref{claim:rank-Eij}. Let $T = \begin{bmatrix}T_V&T_{\ann}\end{bmatrix}\in\M(n,\F)$. 
    Since $T_V$ is an $n\times r$ matrix of rank $r$, we can find a permutation matrix $P\in\GL(n,\F)$ such that the first $r$ rows and columns of $PT_V$ form an invertible matrix. 
    We can extend any basis of $P(V)$ with $(n-r)$ linearly independent vectors to span $\F^n$. Specifically, the full basis can be represented as an invertible matrix $M\in\GL(n,\F)$ of which the first $r$ columns form $PT_V$. Now break $M$ into blocks:
    $$
    M=
    \begin{bmatrix}
        A&B\\
        C&D
    \end{bmatrix},
    $$
    where $A\in\M(r,\F)$, $B\in\M(r\times(n-r),\F)$, $C\in\M((n-r)\times r,\F)$ and $D\in\M(n-r,\F)$. Note that $PT_V=\smat{A\\C}$. It follows that $A$ is invertible. Similarly, break $M^{-1}$ into blocks:
    $$
    M^{-1}=
    \begin{bmatrix}
        A^{\prime}&B^{\prime}\\
        C^{\prime}&D^{\prime}
    \end{bmatrix},
    $$
    where $A^{\prime}\in\M(r,\F)$, $B^{\prime}\in\M(r\times(n-r),\F)$, $C^{\prime}\in\M((n-r)\times r,\F)$ and $D^{\prime}\in\M(n-r,\F)$. Since $D'$ is the Schur complement of $A$,
    $D^{\prime}$ is also invertible. Note that
    $$
    \begin{bmatrix}C^{\prime}&D^{\prime}\end{bmatrix}PT_V=\begin{bmatrix}C^{\prime}&D^{\prime}\end{bmatrix}\begin{bmatrix}A\\C\end{bmatrix}=0.
    $$
    Since $\rk(\begin{bmatrix}C^{\prime}&D^{\prime}\end{bmatrix}P)=n-r$, it follows that the rows of $\begin{bmatrix}C^{\prime}&D^{\prime}\end{bmatrix}P$ form a basis of $\ann$. So we let $T_{\ann}=(\begin{bmatrix}C^{\prime}&D^{\prime}\end{bmatrix}P)^t=P^{-1}\smat{{C^{\prime}}^t\\{D^{\prime}}^t}$ and thus,
    $$
    T=
    \begin{bmatrix}
        T_V & T_{\ann}
    \end{bmatrix}
    =P^{-1}
    \begin{bmatrix}
        A&{C^{\prime}}^t\\C&{D^{\prime}}^t
    \end{bmatrix}.
    $$
    Let $W=\{P^{-1}(i):~i\in[r]\}$. Then $W\subseteq \supp(V)$ and $[n]\setminus W=\{P^{-1}(i):~i\in[n]\setminus[r]\}\subseteq \supp(\ann)$. By~\cref{claim:rank-Eij}, $\rk(\E_{i,j}|_{\ann, V})=1$ if and only if $i\in\supp(\ann)$ and $j\in\supp(V)$. On the other hand, $\partial W=\{\{i,j\}\in E:~i\in W,~j\in[n]\setminus W\}$. Thus $\sum_{\{i,j\}\in E}\rk(\E_{i,j}|_{\ann, V})\geq |\partial W|$. 

    In short, for every subspace $V\leq \F^n$ of dimension $1 \leq r\leq n/2$, we can find a set $W$ of $r$ vertices such that 
    \begin{equation}\label{eq: dim edge exp greater than edge exp}
    	\frac{\sum_{\{i,j\}\in E}\rk(\E_{i,j}|_{\ann, V})}{d\cdot \dim(V)}\geq \frac{|\partial W|}{d|W|}.
    \end{equation}
    Ranging over all subspace $V$ of dimension at most $n/2$ implies that $h_D(\vB_G)\geq h(G)$. The reverse inequality $h(\vB_G)\leq h(G)$ follows by simply choosing $V$ to be the coordinate subspace $\langle e_i\rangle_{i \in W}$, for which it is clear that $\sum_{\{i,j\} \in E} \rk(\E_{i,j}|_{\ann,V})=\ab{\partial W}$.
\end{proof}

We now turn to vertex and dimension expansion.
The following proof is based on the ideas of \cite{DS09, DW10}:
\begin{proposition}\label{prop:graphical-vertex}
    For any $d$-regular graph $G=([n],E)$, we have $\mu(\vB_G)=\mu(G)$.
\end{proposition}
\begin{proof}
    For a non-zero vector $v \in \mathbb{F}^n$, denote by $\pi(v) \in[n]$ the largest index of a non-zero coordinate of $v$. Similarly, let $\pi(V)=\{\pi(v) \mid v \in V \setminus \{0\}\}$. For a set $S\subseteq[n]$ and $i,j\in[n]$, we define $f_{i,j}(S)=\{i\}$ if $j\in S$ and $f_{i,j}(S)=\varnothing$ otherwise. For any $V\leq\F^{n}$, we claim that
    \begin{equation}\label{eq:1}
        \pi(\E_{i,j}(V))\supseteq f_{i,j}(\pi(V)),
    \end{equation}
    where equality holds if $V$ is a coordinate subspace. To see this, it suffices to consider the following two cases:
    \begin{itemize}
        \item If there exists $v\in V$ such that the $j$th coordinate of $v$ is non-zero, then $\E_{i,j}(V)=\langle e_i\rangle$, which implies $\pi(\E_{i,j}(V))=\{i\}\supseteq f_{i,j}(\pi(V))$. Moreover, if $V$ is a coordinate subspace, we can also conclude that $e_j\in V$ and thus $j\in \pi(V)$, which implies that $\pi(\E_{i,j}(V))=\{i\}= f_{i,j}(\pi(V))$.
        \item If there doesn't exist $v\in V$ such that the $j$th coordinate of $v$ is non-zero, then $\E_{i,j}(V)=\{0\}$ and $j\notin\pi(V)$, which implies $\pi(\E_{i,j}(V))=\varnothing$ and $f_{i,j}(\pi(V))=\varnothing$.
    \end{itemize}
    Furthermore, for any $V_1,V_2\leq \F^n$, we have that
    \begin{equation}\label{eq:2}
        \pi(V_1+ V_2)\supseteq\pi(V_1)\cup\pi(V_2),
    \end{equation}
    where equality holds if $V_1$ and $V_2$ are coordinate subspaces or if one of them only consists of the zero vector. Thus, for any subspace $V\leq\F^n$ of dimension $\leq n/2$, we see that
    \begin{align*}
        |\pi(V+\vB_G(V))|&\geq|\pi(V)\cup\pi(\vB_G(V))|\quad &[\text{by \eqref{eq:2}}]\\
        &=\left|\pi(V)\cup\pi\left(\sum_{\{i,j\}\in E}\E_{i,j}(V)\right)\right|\quad &\\
        &=\left|\pi(V)\cup\left(\bigcup_{\{i,j\}\in E}\pi(\E_{i,j}(V))\right)\right|&[\text{by \eqref{eq:2} and }\E_{i,j}(V) = \langle e_i\rangle\text{ or }\{0\}]\\
        &\geq\left|\pi(V)\cup\left(\bigcup_{\{i,j\}\in E}f_{i,j}(\pi(V))\right)\right|&[\text{by \eqref{eq:1}}].
    \end{align*}
     Observe that $|\pi(V)|=\dim(V)\leq n/2$, as we can always find a basis of $V$ with distinct last non-zero coordinates. Additionally, for each vertex $j\in\pi(V)$, we have that $\cup_{\{i,j\}\in E}f_{i,j}(j)$ is the set of neighbors of $j$ in $G$. It follows that
     $$
     \pi(V)\cup\left(\bigcup_{\{i,j\}\in E}f_{i,j}(\pi(V))\right)=\pi(V)\cup\left(\partial_{out}(\pi(V))\right).
     $$
     Therefore, by the definition of vertex expansion,
     $$
    |\pi(V+\vB_G(V))|\geq \left|\pi(V)\cup\left(\bigcup_{\{i,j\}\in E}f_{i,j}(\pi(V))\right)\right|=\left|\pi(V)\cup\left(\partial_{out}(\pi(V))\right)\right|\geq (1+\mu(G))\cdot|\pi(V)|.
     $$
     Therefore, we conclude that
     \[
     \mu(\vB_G)=\min_{\substack{V\leq \F^n \\1\leq\dim(V)\leq n/2}} \frac{\dim(V+\vB_G(V))-\dim(V)}{\dim(V)}\geq \mu(G).
     \]
     The reverse inequality follows by picking $V$ to be a coordinate subspace, which turns all the inequalities above into equalities.
\end{proof}
\begin{remark}
	The proof of \cref{prop:graphical-vertex} actually works for any graph, and \cref{prop:graphical-edge} also holds for any graph after removing the $d$-normalization from the definitions of edge expansion and dimension edge expansion.
\end{remark}

\section{Conclusion and open problems}\label{sec:conclusion}

Recall \eqref{eq: quantum edge expansion norm formulation}, which gives an equivalent definition of the quantum edge expansion of a doubly stochastic matrix tuple in terms of the Schatten-$2$ norm, namely
\[
    h_Q(\vB)=\min_{\substack{V \leq \bC^n\\1\leq\dim(V)\leq\frac{n}{2}}}\frac{\sum_{i=1}^d\norm{B_i|_{V^\perp,V}}_{S_2}^2}{d\cdot\dim(V)}.
\]
Given this formulation, the following definition is natural.
\begin{definition}
    Given $p \in [1,\infty)$ and a doubly stochastic matrix tuple $\vB=(B_1,\dots,B_d) \in \M(n,\bC)^d$, the \emph{Schatten-$p$ edge expansion} of $\vB$ is defined as
    \[
        h_{S_p}(\vB) \coloneqq \min_{\substack{V \leq \bC^n\\1\leq\dim(V)\leq\frac{n}{2}}}\frac{\sum_{i=1}^d\norm{B_i|_{V^\perp,V}}_{S_p}^p}{d\cdot\dim(V)}.
    \]
\end{definition}
The proof of \cref{thmit:edge-comparison} immediately shows that for any doubly stochastic matrix tuple $\vB \in \M(n,\bC)^d$ and for any $p \in [1,\infty)$, we have
\[
    h_{S_p}(\vB) \leq d^{\frac p2}\cdot h_D(\vB).
\]
In case $\vB$ is a unitary matrix tuple, we have the stronger inequality
\[
    h_{S_p}(\vB) \leq h_D(\vB).
\]
Indeed, to prove both of these, we simply recall that $\norm{B_i|_{V^\perp,V}}_{S_p}^p$ is the sum of the $p$th powers of the singular values of $B_i|_{V^\perp,V}$. There are $\rk(B_i|_{V^\perp,V})$ non-zero singular values, and each of them is upper-bounded by the operator norm of $B_i|_{V^\perp,V}$. This operator norm, in turn, is upper-bounded by $\sqrt d$, and by $1$ in case $B_i$ is unitary.

Therefore, for any $p \in [1,\infty)$, Schatten-$p$ edge expansion implies dimension edge expansion, and thus dimension expansion. On the other hand, one can modify the proof of \cref{thm:main-negative} to show that the converse does not hold for any $p \in [1,\infty)$. Indeed, if $\vU \in \U(n)^d$ is a unitary matrix tuple, then it is easy to see that $h_{S_p}(\vU^s) \to 0$ as $s \to 0$, since the tuple $\vU^s$ converges to the identity tuple $(I_n,\dots,I_n)$ as $s \to 0$. However, \cref{thm:exp-keeps-dim-exp} states that $\mu(\vU^s)\geq \mu(\vU)/d$ for all sufficiently small $s$, and thus $h_D(\vU^s)$ stays bounded away from zero as $s \to 0$.

Given this, it is very natural to ask whether the notions of Schatten-$p$ edge expansion are all equivalent.
\begin{question}
    Fix $p,q \in [1,\infty)$ and $d \in \N$. Do there exist increasing functions $f,g:\R_{\geq 0} \to \R_{\geq 0}$ such that
    \[
        f(h_{S_p}(\vB)) \leq h_{S_q}(\vB) \leq g(h_{S_p}(\vB))
    \]
    holds for all doubly stochastic matrix tuples $\vB \in \M(n,\bC)^d$?
\end{question}
If the answer is positive, this could be viewed as a linear-algebraic analogue of a theorem of Matou\v sek \cite{Matousek97}, who proved that a certain $L^p$ notion of graph expansion is equivalent to spectral expansion (i.e.\ the $L^2$ notion) for all $p \in [1,\infty)$.

\bibliographystyle{alpha}
\bibliography{references}

\end{document}